\documentclass{article}
\usepackage{amsmath,enumerate,amssymb,amsthm}
\newtheorem{lemma}{Lemma}
\newtheorem{corollary}{Corollary}
\newtheorem{theorem}{Theorem}
\begin{document}
\def\rc#1{\frac{1}{#1}}
\def\cases#1{\left\{\begin{matrix} #1 \end{matrix}\right.}
\def\pmatrix#1{\left(\begin{matrix} #1 \end{matrix}\right)}
\def\nonpmatrix#1{\begin{matrix} #1 \end{matrix}}
%
\title{On the Structure of Additive Systems of Integers}%
\author{M. N. Huxley, M. C. Lettington and K. M. Schmidt \\ School of Mathematics, Cardiff University\\ Cardiff CF24 4AG, UK}

\maketitle
\begin{abstract}
A sum-and-distance system is a collection of finite sets of integers such that
the sums and differences formed by taking one element from each set generate
a prescribed arithmetic progression.
Such systems, with two component sets, arise naturally in the study of matrices
with symmetry properties and consecutive integer entries.
Sum systems are an analogous concept where only sums of elements are considered.
\par
We establish a bijection between sum systems and sum-and-distance systems of
corresponding size, and show that sum systems are equivalent to principal
reversible cuboids, which are tensors with integer entries and a symmetry of
`reversible square' type.
We prove a structure theorem for principal reversible cuboids, which gives rise
to an explicit construction formula for all sum systems in terms of joint
ordered factorisations of their component set cardinalities.
\end{abstract}

\section{Introduction}
\label{sIntro}
The present paper concerns the relationship between sum-and-distance systems
and sum systems, and their general structure, including a construction method
for all such systems. Roughly speaking, a sum-and-distance system consists of
several component sets of natural numbers such that the sums comprising one
element, or its negative, of each set generate a prescribed target
set, specifically an arithmetic progression.
More simply, a sum system consists of several sets of non-negative integers
such that the sums formed by choosing exactly one term from each set generate
a sequence of consecutive integers.
For the precise definitions, see Section \ref{sDefi} below.
\par
Two-component sum-and-distance systems arise naturally when we consider
square arrays of consecutive integers with certain symmetry properties.
The algebraic properties of square matrices with different types of symmetries
were recently explored in \cite{rSupAlg}, also giving construction formulae for the
various types. In that paper, the matrix entries were assumed to be general
real numbers, allowing the symmetry classes to form direct summands in a
${\mathbb Z}_2$-graduation of the matrix algebra over ${\mathbb R}$.
However, an additional level of complication is introduced when we require
the matrix entries to be integers or, more specifically, a consecutive
sequence of integers, such as in a magic square or a principal reversible square \cite{rOB}.
\par
A reversible square
$M$
is an
$n \times n$
matrix with the properties of column and line reversal symmetry (R) and the
vertex sum property (V) (see \cite{rSupAlg} and equations (\ref{eRcub}), (\ref{eVdef}) in Section \ref{sprc} below).
Such a matrix also has the associated symmetry that any two entries in
diametrically opposite positions with respect to the centre of the matrix add
up to the same constant
$2 w$
(see \cite{rSupAlg}, Lemma 7.1). Subtracting
$w$
from each matrix entry, we obtain a reversible square $M_0$ whose entries sum up to
$0.$
If
$n = 2 \nu$
is even,
it is then of the form
(\cite{rSupAlg} Theorem 7.2)
\begin{align}
 M_0 &= \rc 2 \pmatrix{ J(1_\nu\,a^T + b\,1_\nu^T) J & J (-1_\nu\,a^T + b\,1_\nu^T) \cr
 (1_\nu\,a^T - b\,1_\nu^T) J & -1_\nu\,a^T - b\,1_\nu^T},
\label{ewtless}\end{align}
where
$1_\nu \in {\mathbb R}^\nu$
is the vector with all entries equal to 1,
$J \in {\mathbb R}^{\nu \times \nu}$
is the matrix which has entries 1 on the anti-diagonal and 0 elsewhere, and
$a, b \in {\mathbb R}^\nu$
are arbitrary vectors.
If the reversible square
$M$
is to contain exactly the integers
$1, \dots, n^2,$
then the weight
$w$
can be found as the average of all entries,
$w = (n^2 + 1)/2;$
hence the weightless reversible square
$M_0$
will have as entries the numbers
\begin{align}
 \left\{-\frac{n^2-1}2, -\frac{n^2-1}2 + 1, \dots, \frac{n^2-1}2 - 1, \frac{n^2-1}2 \right\}.
\nonumber\end{align}
Considering that multiplication by
$J$
on the left or right just inverts the order of the rows or columns of a matrix,
respectively, we find from (\ref{ewtless}) that the sums
$\pm a_j \pm b_k,$
with
$j, k \in \{1, \dots, \nu\}$
and independently chosen signs, must generate each odd number from
$-n^2 + 1$
to
$n^2 - 1$
exactly once. In other words, the sets of entries of the vectors
$a$
and
$b$
form a two-component non-inclusive sum-and-distance system as defined in
Section \ref{sDefi} below.
\par
If
$n = 2 \nu + 1$
is odd, then the weightless reversible square
$M_0$
will have the form (\cite{rSupAlg} Theorem 7.2)
\begin{align}
 M_0 &= \pmatrix{J(1_\nu\,a^T + b\,1_\nu^T)J & J b & J(-1_\nu\,a^T + b\,1_\nu^T) \cr
 (J a)^T & 0 & - a^T \cr
 (1_\nu\,a^T - b\,1_\nu^T)J & - b & -1_\nu\,a^T - b\,1_\nu^T }
\nonumber\end{align}
with vectors
$a, b \in {\mathbb R}^\nu,$
and by the same reasoning as above, we find that, for the reversible square
$M$
to contain the integers
$1, \dots, n^2,$
the sums
$\pm a_j \pm b_k,$
where
$j, k \in \{1, \dots, \nu\}$
and the signs are chosen independently, taken together with the entries
$\pm a_j, \pm b_j$
$(j \in \{1, \dots, \nu\}),$
must generate exactly the integers
$1, 2, \dots, (n^2-1)/2$
and their negatives. In other words, the sets of entries of the vectors
$a$
and
$b$
form a two-component inclusive sum-and-distance system, as defined in
Section \ref{sDefi} below.
\par
Similarly, sum-and-distance systems appear in a certain type of rank 3
associated magic squares (cf.\ \cite{rBloRep} Theorem 9).
As shown in Lemma 3.1 and Theorem 4.1 of \cite{rSupAlg} (see also \cite{rBloRep} Theorem 1), a
$2\nu \times 2 \nu$
matrix
$M$
will have all rows and columns adding up to the same number, and also the
associated symmetry described above, if, after subtracting the weight
$w,$
it has the form
\begin{align}
 M_0 &= \rc 2 \pmatrix{J(J V^T + W J)J & J(-J V^T + W J) \cr
   (J V^T - W J)J & -J V^T - W J}
\nonumber\end{align}
with matrices
$V, W \in {\mathbb R}^{\nu \times \nu}$
whose rows add up to 0. Specifically, if
$\nu$
is even and we choose vectors
$v, w$
with entries
$\pm 1$
which, for each vector, add up to 0, and further vectors
$a, b \in {\mathbb R}^\nu,$
and set
$V = a\,v^T,$
$W = b\,w^T,$
then the resulting matrix
$M$
(after adding the weight $w = (n^2+1)/2$)
will be an associated magic square with entries
$1, \dots, n^2$
if and only if the sets of entries of
$a$
and
$b$
form a two-component non-inclusive sum-and-distance system.
\par
As a final example, we mention most perfect squares; these are square matrices
of even dimensions which, in addition to having all rows and columns adding up
to the same number, have the properties that all
$2 \times 2$
submatrices have the same sum of entries, and that all pairs of entries half
the matrix size apart on any diagonal add up to the same number.
By \cite{rSupAlg} Theorem 6.2 any
$2 \nu \times 2 \nu$
most perfect square, with even
$\nu,$
is, after subtracting the weight from each entry, of the form
\begin{align}
 M_0 &= \pmatrix{a\,\S_\nu^T + \S_\nu\,b^T & a\,\S_\nu^T - \S_\nu\,b^T \cr
  -a\,\S_\nu^T + \S_\nu\,b^T & -a\,\S_\nu^T - \S_\nu\,b^T},
\nonumber\end{align}
where
$a, b \in {\mathbb R}^\nu$
are any vectors and
$\S_\nu = (1, -1, 1, -1, \dots, 1, -1)^T \in {\mathbb R}^\nu.$
Again we see that
$M$
will have entries
$1, \dots, (2 \nu)^2$
if and only if the sets of entries of the vectors
$2 a$
and
$2 b$
form a two-component non-inclusive sum-and-distance system.
\par
Sum systems are conceptually simpler. They are directly related to
reversible cuboids, the multidimensional analogues of reversible squares and
rectangles, as shown in Theorem \ref{tssten} below.
Further, it is one of the results of the present study that sum systems are
in one-to-one correspondence with sum-and-distance systems (Theorems \ref{tSdsSs1}, \ref{tSdsSs2}).
\par
We remark that sum systems can be interpreted as discrete local coordinate
systems for a set of consecutive integers, generalising the base
$q$
decimal representation. Indeed, the integers
$0, 1, \dots, q^m-1$
can be uniquely represented in the form
\begin{align}
 \sum_{j=1}^m a_j q^{j-1},
\nonumber\end{align}
where
$a_j \in \{0, 1, \dots, q-1\},$
so the sets
\begin{align*}
&\{0, 1, 2, \dots, q-1\}, \{0, q, 2q, \dots, q^2-q\}, \{0, q^2, 2q^2, \dots, q^3-q^2\},\\
&\qquad\qquad\dots,
\{0, q^{m-1}, 2 q^{m-1}, \dots, q^m - q^{m-1}\}
\end{align*}
form an $m$-component sum system in the sense defined in Section \ref{sDefi} below.
Using this system as a basis, the $m$ entries, one taken from each component
set, which add up to a given number can be considered as that number's
discrete coordinates.
In general, sum systems will have a considerably more complicated structure
than the above simple arithmetic progressions, and it is one of the main
results of the present paper to provide a constructive description of the
general sum system (see Theorem \ref{tssbuild}).
\par
Research on some related topics has been undertaken previously, including
the study of arithmetic progressions arising in the sum of two sets of integers
\cite{rB}, \cite{rgreen}; comparing the sizes of the sum set and the difference set of a set with itself \cite{rMO}, \cite{rR}, \cite{rN};
for an overview of this subject, see \cite{rGr}.
However, it seems that despite the simplicity of the concepts, sum systems and sum-and-distance
systems, as studied here, have not attracted much attention in the mathematical
literature, and our present results are new.
\par
The paper is organised as follows.
After giving the definitions of sum-and-distance systems and sum systems in
Section \ref{sDefi}, we use a polynomial factorisation method to show in Section
\ref{sSdsSs} that there is a one-to-one relationship between sum-and-distance
systems and sum systems of suitable size.
It is fairly straightforward to see that a sum-and-distance system generates
a corresponding sum system, but the fact that every sum system arises in this
way is not obvious.
In Section \ref{sprc}, we explore the connection between $m$-component sum systems and
$m$-dimensional principal reversible cuboids, which are generalisations of Ollerenshaw and
Br\'ee's principal reversible squares \cite{rOB} from square matrices to more
general order $m$ tensors. This shows that the structure of sum systems (and
hence, by means of the bijection, of sum-and-distance systems) can be fully
understood in terms of the construction of principal reversible cuboids.
In Section \ref{sconstr}, we establish that the structure of the latter is
essentially recursive, in the sense that any principal reversible cuboid
arises from glueing offset copies of a maximal principal reversible subcuboid
together. Finally, in Section \ref{sjof} we show that, due to this recursive
property, every principal reversible cuboid can be constructed by means
of a chain of building operators with parameters arising from a joint ordered
factorisation of the cuboid's dimensions, thus linking the structure of
principal reversible cuboids with number theoretic properties of their sizes.
As a result, we obtain the general structure of the component sets of sum
systems as nested arithmetic progressions.
We conclude with some examples which illustrate how sum systems and
sum-and-distance systems arise from joint ordered factorisations.
\section{Definition of sum-and-distance systems and sum systems}
\label{sDefi}
Arithmetic progressions play a central role in the present paper.
We use the notation
\def\ap#1{\langle #1 \rangle}
\def\Ap#1{\left\langle #1 \right\rangle}
$\ap m := \{0, 1, \dots, m-1\}$
for any
$m \in {\mathbb N},$
so the arithmetic progression with start value
$a,$
step size
$s$
and
$N$
terms can be expressed as
$s \ap N + a \ (= \{a, a+s, a+2s, \dots, a+(N-1)s \}).$
\par
Note that we use the standard convention that
$A + B = \{x + y : x \in A, y \in B\}$
and
$a A + b = \{a x + b : x \in A\}$
for sets
$A, B \subset {\mathbb R}$
and
$a, b \in {\mathbb R}$
throughout.
As usual,
$A - B = A + (-B).$
We write
$|M|$
for the cardinality of a finite set
$M.$
\par\medskip
{\it Definition.}
Let
$\nu, \mu \in {\mathbb N}.$
We call a pair of sets
$\{a_1, \dots, a_\nu\}, \{b_1, \dots, b_\mu\} \subset {\mathbb N}$
a
{\it (non-inclusive) sum-and-distance system\/}
if
\begin{align}
 \{|a_j \pm b_k| : j \in \{1, \dots, \nu\}, k \in \{1, \dots, \mu\}\} &= 2 \ap{2 \nu \mu} + 1.
\nonumber\end{align}
The set of pairs is called an
{\it inclusive sum-and-distance system\/}
if
\begin{align}
 \{|a_j \pm b_k|, a_j, b_k : j \in \{1, \dots, \nu\}, k \in \{1, \dots, \mu\}\} &= \ap{2 \nu \mu + \nu + \mu} + 1.
\nonumber\end{align}
\par\medskip\noindent
The target set for a non-inclusive sum-and-distance system,
$2 \ap{2 \nu \mu} +1 = \{1, 3, 5, \dots, 4 \nu \mu -1\},$
differs from that for an inclusive sum-and-distance system,
$\ap{2 \nu \mu + \nu + \mu} + 1 = \{1, 2, \dots, 2 \nu \mu + \nu + \mu\},$
in that the former only has odd integers; this difference is
motivated by the situations outlined above in which sum-and-distance systems
arise, and the reason for it will be made transparent by Theorems \ref{tSdsSs1}, \ref{tSdsSs2}.
\par
Sum-and-distance systems can be equivalently characterised by a target set of
positive and negative numbers in the following way.
\begin{lemma}\label{lfulltarget}
Let
$\{a_1, \dots, a_\nu\}, \{b_1, \dots, b_\mu\} \subset {\mathbb N},$
$\nu, \mu \in {\mathbb N}.$
\par\medskip\noindent
(a)
These sets form a non-inclusive sum-and-distance system if and only if
\begin{align}
 \{\pm a_j \pm b_k : j \in \{1, \dots, \nu\}, k \in \{1, \dots, \mu\}\} = 2 \ap{4 \nu \mu} - 4 \nu \mu + 1,
\nonumber\end{align}
where the signs
$\pm$
are chosen independently, so there are 4 elements of the set for each pair
$(j, k).$
\par\medskip\noindent
(b)
These sets form an inclusive sum-and-distance system if and only if
\begin{align}
 \{\pm a_j \pm b_k, \pm a_j, \pm b_k, 0 : j \in \{1, \dots, \nu\}, &k \in \{1, \dots, \mu\}\}\nonumber \\
  &= \ap{(2 \nu + 1) (2 \mu + 1)} - 2 \nu \mu - \nu - \mu,
\nonumber\end{align}
where the signs
$\pm$
are chosen independently, so there are 8 elements of the set for each pair
$(j, k).$
\end{lemma}

\begin{proof}
(a)
The sums
$\pm a_j \pm b_k$
will give exactly the sums and absolute distances
$|a_j \pm b_k|$
and their negatives
$-|a_j \pm b_k|,$
so the resulting set will be the union of the target set of the non-inclusive sum-and-distance system with its negative; this can be written as the step-2
arithmetic progression on the right-hand side.
\par\medskip\noindent
(b)
The sums
$\pm a_j \pm b_k$
give the same results as in (a), and including the elements
$\pm a_j$
and
$\pm b_k,$
we obtain the union of the target set of the inclusive sum-and-distance system with its
negative. By adding the element 0 to the set, we can complete this to the
arithmetic progression on the right-hand side.
\end{proof}

The above lemma motivates the following generalisation.
\par\medskip
{\it Definition.}
Let
$m \in {\mathbb N}$
and
$A_j \subset {\mathbb N}$
$(j \in \{1, \dots, m\}).$
Then we call
$A_1, A_2, \dots, A_m$
a
{\it (non-inclusive) $m$-part sum-and-distance system\/}
if
\begin{align}
 \sum_{j=1}^m (A_j \cup (-A_j)) &= 2 \Ap{2^m \prod_{j=1}^m |A_j|} - 2^m  \prod_{j=1}^m |A_j| + 1.
\nonumber\end{align}
We call
$A_1, A_2, \dots, A_m$
an
{\it inclusive $m$-part sum-and-distance system\/}
if
\begin{align}
 \sum_{j=1}^m (A_j \cup \{0\} \cup (-A_j)) &= \Ap{\prod_{j=1}^m (2|A_j|+1)} - \rc 2 \left(\prod_{j=1}^m (2|A_j|+1) - 1 \right).
\nonumber\end{align}
\par\medskip
{\it Definition.}
Let
$n_1, n_2 \in {\mathbb N}+1.$
\hfill\break
We call a pair of sets
$A = \{a_1, \dots, a_{n_1}\}, B = \{b_1, \dots, b_{n_2}\} \subset {\mathbb N}_0$
a
{\it sum system\/}
if
\begin{align}
 A + B &= \ap{n_1 n_2},
\nonumber\end{align}
i.e. in explicit form,
\begin{align}
 \{a_j + b_k : j \in \{1, \dots, n_1\}, k \in \{1, \dots, n_2\}\} &= \{0, 1, \dots, n_1 n_2 - 1\}.
\nonumber\end{align}
More generally, we call a collection of
$m$
sets
$A_1, A_2, \dots, A_m \subset {\mathbb N}_0,$
each of cardinality at least 2,
an
{\it $m$-part sum system\/}
if
\begin{align}
 \sum_{k=1}^m A_k &= \Ap{\prod_{k=1}^m \left| A_k \right|};
\nonumber\end{align}
note that
\begin{align}
 \sum_{k=1}^m A_k = \left\{ \sum_{k=1}^m a_k : a_k \in A_k \ (k \in \{1, \dots, m\})\right \}.
\nonumber\end{align}
\par\medskip\noindent
Since the number 0 in the target set
can only arise as a sum of 0s, as all numbers in the sets are non-negative,
it follows that each component set of a sum system contains the number 0.
\section{Correspondence between sum-and-distance systems and sum systems}
\label{sSdsSs}
Given a finite set
$M \subset {\mathbb N}_0,$
we can associate with it the polynomial
\begin{align}
 p_M(x) &= \sum_{j \in M} x^j.
\label{epoly}\end{align}
More generally, for a finite set
$M \in {\mathbb Z},$
we have an associated Laurent polynomial (\ref{epoly}) which may include negative
powers.
\par
Specifically for the arithmetic progression
$M = s \ap N + a,$
where
$s, N \in {\mathbb N}$
and
$a \in {\mathbb N}_0,$
we find
\begin{align}
 p_{s \ap N + a}(x) &= \sum_{j=0}^{N-1} x^{a+j s} = x^a \sum_{j=0}^{N-1} (x^s)^j.
\nonumber\end{align}
Clearly this polynomial has root
$0$
with multiplicity
$a;$
it is also evident that
$1$
is not a root, nor are the other $s$th roots of unity.
Hence, to find the further roots of this polynomial, we may assume
$x^s \neq 1$
and observe that
\begin{align}
 p_{s \ap N + a}(x) &= x^a \sum_{j=0}^{N-1} (x^s)^j = x^a\,\frac{1 - (x^s)^N}{1 - x^s} = x^a\,\frac{1 - x^{s N}}{1 - x^s},
\nonumber\end{align}
which shows that the non-zero roots of
$p_{s \ap N + a}$
are exactly the ($s N$)th roots of unity which are not $s$th roots of unity; in particular, they
all lie on the complex unit circle.
\par\medskip
{\it Definition.}
Let
$m \in {\mathbb N}.$
A polynomial
$p$
of degree
$d$
is called
{\it palindromic\/}
if
it is equal to its reciprocal polynomial, i.e.\ if
$p(x) = x^{d} p(\rc x),$
so
\begin{align}
 p(x) = \sum_{j=0}^{d} \alpha_j\, x^j
\nonumber\end{align}
with
$\alpha_j = \alpha_{d - j}$
$(j \in \{0, \dots, d\}).$
\par\medskip\noindent
The results of this section will rely on the following key observation.
\begin{lemma}\label{lkey}
Let
$p$
be a polynomial with real coefficients and with all its roots situated on the complex
unit circle.
\begin{description}
\item{(a)}
If all roots of
$p$
are non-real, then
$p$
is palindromic and of even degree.
\item{(b)}
If all roots of
$p$
are non-real except for the simple root
$-1,$
then
$p$
is palindromic and of odd degree.
\end{description}
\end{lemma}
\begin{proof}
(a)
As the polynomial has real coefficients, its (non-real) roots come in complex
conjugate pairs, say
$\{r_j, \overline{r_j} \mid j \in \{1, \dots, m\}\}.$
Thus
\begin{align}
 p(x) &= \prod_{j=1}^m (x - r_j) (x - \overline{r_j})
 = x^{2m} \prod_{j=1}^m \left(1 - \frac{r_j}x\right) \left(1 - \frac{\overline{r_j}}x\right)
\nonumber\\
 &= x^{2m} \prod_{j=1}^m r_j\,\overline{r_j}\,\left(\rc{r_j} - \rc x\right) \left(\rc{\overline{r_j}} - \rc x\right)
 = x^{2m} \prod_{j=1}^m \left(\rc x - \overline{r_j}\right) \left(\rc x - r_j\right)
\nonumber\\
 &= x^{2m}\,p\left(\rc x\right),
\nonumber\end{align}
with
$2m$
the degree of the polynomial, so
$p$
is palindromic.
\par\medskip\noindent
(b)
The polynomial
$p$
can be factorised as
$p(x) = (x + 1) \tilde p(x),$
where
$\tilde p$
only has non-real roots situated on the unit circle. Writing
\begin{align}
 p(x) = \sum_{j=0}^{d} \alpha_j\,x^j, \qquad & \tilde p(x) = \sum_{j=0}^{d-1} \tilde \alpha_j\,x^j,
\nonumber\end{align}
where
$d$
is the degree of the polynomial
$p,$
a straightforward calculation gives
\begin{align}
 \alpha_j &= \cases{\tilde \alpha_0 & if j=0, \cr \tilde \alpha_j + \tilde \alpha_{j+1} & if j \in \{1, \dots, d-1\}, \cr \tilde \alpha_{d-1} & if j = d;}
\label{ecoefform}\end{align}
and hence it follows by recursion that
$\tilde \alpha_j \in {\mathbb R}$
$(j \in \{0, \dots, d-1\}),$
since
$p$
has real coefficients. Therefore we can apply part (a) to find that
$\tilde p$
is palindromic of even degree, i.e.
$\tilde \alpha_j = \tilde \alpha_{d-1-j}$
$(j \in \{0, \dots, d-1\}).$
Hence
$p$
is of odd degree, and we deduce from (\ref{ecoefform}) that
\begin{align}
 \alpha_{d} &= \tilde \alpha_{d-1} = \tilde \alpha_0 = \alpha_0,
\nonumber\\
 \alpha_j &= \tilde \alpha_j + \tilde \alpha_{j-1} = \tilde \alpha_{d-1-j} + \tilde \alpha_{d-1-j+1} = \alpha_{d-j}
 \quad (j \in \{1, \dots, d-1\}),
\nonumber\end{align}
so
$p$
is palindromic.
\end{proof}
Using this result, we can show that the component sets of sum systems
always have a palindromic structure, too, in the following sense.
\begin{theorem}\label{tsspalin}
Let
$m \in {\mathbb N}.$
Suppose the sets
$A_1, A_2, \dots, A_m \subset {\mathbb N}_0$
form an $m$-part sum system. Then, for each
$j \in \{1, \dots, m\},$
\begin{align}
 A_j = (\max A_j) - A_j,
\nonumber\end{align}
i.e.
$x \in A_j$
if and only if
$(\max A_j - x) \in A_j.$
\par
Moreover,
if all component sets
$A_j$
have odd cardinality, then
$\max A_j$
is even for every
$j \in \{1, \dots, m\};$
if at least one component set has even cardinality, then
$\max A_j$
is odd for exactly one
$j \in \{1, \dots, m\}.$
\end{theorem}
\begin{proof}
Denoting the elements of the set
$A_j$
by
$a^{(j)}_1, a^{(j)}_2, \dots, a^{(j)}_{d_j},$
where
$d_j = |A_j|,$
and setting
$d = \prod \limits_{k=1}^m |A_k|,$
we find
\begin{align}
 \prod_{j=1}^m p_{A_j}(x) &= \left(\sum_{k_1=1}^{d_1} x^{a^{(1)}_{k_1}}\right)
 \left(\sum_{k_2=1}^{d_2} x^{a^{(2)}_{k_2}}\right) \cdots
 \left(\sum_{k_m=1}^{d_m} x^{a^{(m)}_{k_m}}\right)
\nonumber\\
 &= \sum_{k_1=1}^{d_1} \sum_{k_2=1}^{d_2} \cdots \sum_{k_m=1}^{d_m} x^{a^{(1)}_{k_1} + a^{(2)}_{k_2} + \cdots + a^{(m)}_{k_m}}
 = \sum_{j=0}^{d-1} x^j
 = \frac{1 - x^d}{1 - x},
\label{esspoly}\end{align}
where we used the sum system property
$\sum \limits_{k=1}^m A_k = \ap d$
in the penultimate step.
This shows that the polynomials
$p_{A_j}$
form a factorisation of the polynomial on the right-hand side of (\ref{esspoly}).
Now we distinguish between two cases.
\par\medskip
{\it 1st case.\/}
If
$d = \prod \limits_{j=1}^m d_j$
is odd, i.e. if all
$d_j$
are odd, then the polynomial on the right-hand side of (\ref{esspoly}) has no
real roots; its roots are the non-real $d$th roots of unity.
Hence, for any
$j \in \{1, \dots, m\},$
$p_{A_j}$
has only non-real roots situated on the complex unit circle, and it has
real coefficients (in fact, coefficients in $\{0, 1\}$). Hence, by Lemma \ref{lkey} (a),
$p_{A_j}$
has even degree
and is palindromic, which gives the stated property for
$A_j.$
\par\medskip
{\it 2nd case.\/}
If
at least one of the component set cardinalities
$d_j$
is even, then
$d$
is even, so
$-1$
is a (simple) root of the polynomial on the right-hand side of (\ref{esspoly}).
Therefore exactly one of the polynomials
$p_{A_j}$
has the root
$-1;$
w.l.o.g. we may assume that
$p_{A_1}$
is this polynomial.
Then for any
$j \in \{2, \dots, m\},$
the same reasoning as in the first case shows that
$p_{A_j}$
has even degree
and is palindromic, while, by Lemma \ref{lkey} (b),
$p_{A_1}$
is palindromic of odd degree.
\end{proof}
This observation allows us to establish the following bijection between
sum-and-distance systems and sum systems.
\begin{theorem}\label{tSdsSs1}
Let
$m \in {\mathbb N}$
and suppose the non-empty sets
$A_1, A_2, \dots, A_m \subset {\mathbb N}$
form an $m$-part non-inclusive sum-and-distance system.
For
$j \in \{1, \dots, m\},$
let
\begin{align}
 \tilde A_j &:= \rc 2\,\max A_j + \rc 2\,(A_j \cup (-A_j)) = \left \{ \frac{(\max A_j) - a}2, \frac{(\max A_j) + a}2 : a \in A_j \right\};
\nonumber\end{align}
then
$\tilde A_1, \tilde A_2, \dots, \tilde A_m$
form an $m$-part sum system, where each part has even cardinality.
\par
Conversely, suppose the sets
$\tilde A_1, \tilde A_2, \dots, \tilde A_m \subset {\mathbb N}_0$
form an $m$-part sum system,
where each component set has even cardinality.
Then, for each
$j \in \{1, \dots, m\},$
denoting the elements of
$\tilde A_j$
by
$0 = \alpha_1 < \alpha_2 < \cdots < \alpha_{2 \nu_j},$
let
\begin{align}
 A_j &:= \{\alpha_{\nu_j + k} - \alpha_{\nu_j + 1 - k} : k \in \{1, \dots, \nu_j\} \};
\nonumber\end{align}
then the sets
$A_1, A_2, \dots, A_m$
form an $m$-part non-inclusive sum-and-distance system.
\end{theorem}
\begin{proof}
We find for the set sum
\begin{align}
 \sum_{j=1}^m \tilde A_j &= \rc 2 \sum_{j=1}^m \max A_j + \rc 2 \sum_{j=1}^m (A_j \cup (-A_j))
\nonumber\\
 &= \rc 2 \sum_{j=1}^m \max A_j + \Ap{2^m \prod_{j=1}^m |A_j|} - 2^{m-1}  \prod_{j=1}^m |A_j| + \rc 2
\nonumber\\
 &= \Ap{2^m \prod_{j=1}^m |A_j|} = \Ap{\prod_{j=1}^m |\tilde A_j|},
\nonumber\end{align}
bearing in mind that
\begin{align}
 \rc 2 \sum_{j=1}^m \max A_j &= 2^m \prod_{j=1}^m |A_j| - 1 - 2^{m-1} \prod_{j=1}^m |A_j| + \rc 2,
\nonumber\end{align}
as the sum of the largest elements of the component sets of a sum-and-distance
system gives the largest element of its target set.
\par
For the converse, we note that for each
$j \in \{1, \dots, m\},$
the component set
$\tilde A_j$
of the sum system has palindromic symmetry by Theorem \ref{tsspalin}, i.e. its ordered elements
satisfy
\begin{align}
 \alpha_{\nu_j+k} + \alpha_{\nu_j+1-k} &= \alpha_{2 \nu_j} \qquad (k \in \{1, \dots, \nu_j\}).
\nonumber\end{align}
Hence
\begin{align}
 \alpha_{\nu_j+k} - \alpha_{\nu_j+1-k} &= 2 \alpha_{\nu_j+k} - \alpha_{2\nu_j}
\nonumber\end{align}
and also
\begin{align}
 -(\alpha_{\nu_j+k} - \alpha_{\nu_j+1-k}) &= 2 \alpha_{\nu_j+1-k} - \alpha_{2 \nu_j},
\nonumber\end{align}
which gives
\begin{align}
 A_j \cup (-A_j) &= \{2 \alpha_{\nu_j+k} - \alpha_{2 \nu_j} : k \in \{1, \dots, \nu\}\} \cup \{2 \alpha_{\nu_j+1-k} - \alpha_{2 \nu_j} : k \in \{1, \dots, \nu\}\}
\nonumber\\
 &= 2 \tilde A_j - \max \tilde A_j.
\nonumber\end{align}
Therefore
\begin{align}
 \sum_{j=1}^m (A_j \cup (-A_j)) &= \sum_{j=1}^m (2 \tilde A_j - \max \tilde A_j)
 = 2 \sum_{j=1}^m \tilde A_j - \sum_{j=1}^m \max \tilde A_j
\nonumber\\
 &= 2 \Ap{\prod_{j=1}^m |\tilde A_j|} - \left(\prod_{j=1}^m |\tilde A_j| - 1 \right)
\nonumber\\
 &= 2 \Ap{2^m \prod_{j=1}^m |A_j|} - 2^m \prod_{j=1}^m |A_j| + 1,
\nonumber\end{align}
as required.
\end{proof}
\begin{theorem}\label{tSdsSs2}
Let
$m \in {\mathbb N}$
and suppose the non-empty sets
$A_1, A_2, \dots, A_m \subset {\mathbb N}$
form an $m$-part inclusive sum-and-distance system. For
$j \in \{1, \dots, m\},$
let
\begin{align}
 \tilde A_j := \max A_j + (A_j \cup \{0\} \cup (-A_j));
\nonumber\end{align}
then
$\tilde A_1, \tilde A_2, \dots, \tilde A_m$
form an $m$-part sum system, where each part has odd cardinality.
\par
Conversely, suppose the sets
$\tilde A_1, \tilde A_2, \dots, \tilde A_m \subset {\mathbb N}_0$
form an $m$-part sum system, where each component set has odd cardinality.
Then, for each
$j \in \{1, \dots, m\},$
denoting the elements of
$\tilde A_j$
by
$0 = \alpha_1 < \alpha_2 < \cdots < \alpha_{2 \nu_j+1},$
let
\begin{align}
 A_j &:= \{\textstyle{\rc 2}\,(\alpha_{\nu_j+1+k} - \alpha_{\nu_j+1-k}) : k \in \{1, \dots, \nu_j\}\};
\nonumber\end{align}
then the sets
$A_1, A_2, \dots, A_m$
form an $m$-part inclusive sum-and-distance system.
\end{theorem}
\begin{proof}
In analogy to the proof of Theorem \ref{tSdsSs1}, we find the set sum
\begin{align}
 \sum_{j=1}^m \tilde A_j &= \sum_{j=1}^m \max A_j + \sum_{j=1}^m (A_j \cup \{0\} \cup (-A_j))
\nonumber\\
 &= \sum_{j=1}^m \max A_j + \Ap{\prod_{j=1}^m (2|A_j|+1)} - \rc 2 \prod_{j=1}^m (2|A_j|+1) + \rc 2
\nonumber\\
 &= \Ap{\prod_{j=1}^m (2|A_j|+1)} = \Ap{\prod_{j=1}^m |\tilde A_j|}.
\nonumber\end{align}
\par
For the converse, we use the fact that for each
$j \in \{1, \dots, m\},$
the component set
$\tilde A_j$
of the sum system has palindromic symmetry by Theorem \ref{tsspalin}, which gives
\begin{align}
 \alpha_{\nu_j+1+k} + \alpha_{\nu_j+1-k} &= \alpha_{2\nu_j+1} \qquad (k \in \{0, \dots, \nu\})
\nonumber\end{align}
(bearing in mind that
$\alpha_1 = 0$),
and in particular
$2 \alpha_{\nu_j+1} = \alpha_{2\nu_j+1}.$
Thus
$\alpha_{2\nu_j+1} = \max \tilde A_j$
is even, which also follows from Theorem \ref{tsspalin}, as all component sets of the sum system have odd cardinality.
Hence
\begin{align}
 {\textstyle \rc 2}\,(\alpha_{\nu_j+1+k} - \alpha_{\nu_j+1-k}) &= \alpha_{\nu_j+1+k} - {\textstyle \rc 2}\,\alpha_{2\nu_j+1} = \alpha_{\nu_j+1+k} - \alpha_{\nu_j+1} \in {\mathbb N},
\nonumber\end{align}
and
\begin{align}
 -{\textstyle \rc 2}\,(\alpha_{\nu_j+1+k} - \alpha_{\nu_j+1-k}) &= \alpha_{\nu_j+1-k} - {\textstyle \rc 2}\,\alpha_{2\nu_j+1} = \alpha_{\nu_j+1-k} - \alpha_{\nu_j+1} \in {\mathbb N}
\nonumber\end{align}
$(k \in \{1, \dots, \nu\}).$
Consequently,
\begin{align}
 A_j &\cup \{0\} \cup (-A_j)
\nonumber\\
 &= \{\alpha_{\nu+1+k} - \alpha_{\nu_j+1} : k \in \{1, \dots, \nu\}\} \cup \{0\} \cup \{\alpha_{\nu+1-k} - \alpha_{\nu_j+1} : k \in \{1, \dots, \nu\}\}
\nonumber\\
 &= \{\alpha_k - \alpha_{\nu_j+1} : k \in \{1, \dots, 2 \nu_j + 1\}\}
 = \tilde A_j - {\textstyle \rc 2} \max \tilde A_j.
\nonumber\end{align}
This gives
\begin{align}
 \sum_{j=1}^m (A_j \cup \{0\} \cup (-A_j)) &= \sum_{j=1}^m \tilde A_j - \rc 2 \sum_{j=1}^m \max \tilde A_j
 = \Ap{\prod_{j=1}^m |\tilde A_j|} - \rc 2 \left(\prod_{j=1}^m |\tilde A_j| - 1 \right)
\nonumber\\
 &= \Ap{\prod_{j=1}^m (2|A_j|+1)} - \rc 2 \left(\prod_{j=1}^m (2|A_j|+1) - 1 \right),
\nonumber\end{align}
proving the claim.
\end{proof}
{\it Remark.\/}
Note that sum systems with odd cardinality throughout correspond to inclusive
sum-and-distance systems, and the tight target set (containing consecutive
integers) of the latter is related to the fact that the maximum of each
component set of the sum system is even, as apparent from the proof of
Theorem \ref{tSdsSs2}.
However, sum systems with even cardinality do not have this property, and
hence their corresponding non-inclusive sum-and-distance systems have a
more sparse target set containing consecutive odd integers only.
Thus the discrepancy between inclusive and non-inclusive sum-and-distance
systems resolves into the simple dichotomy between odd and even cardinality
of the component sets when considering the sum systems.
\par
We remark further that at the level of sum systems, there is no reason to
require that the components have all odd or all even cardinality.
A sum system with mixed parity will, by the transforms given in Theorems
\ref{tSdsSs1} and \ref{tSdsSs2}, correspond to a hybrid inclusive/non-inclusive
sum-and-distance system, but we do not pursue this correspondence further
in the present study.
\section{Principal reversible cuboids and sum systems}
\label{sprc}
In this section we shall extend the definition of reversible square matrices,
which can be considered as order 2 tensors, to general order $m$ tensors.
We use multiindex notation, i.e. tensor components are indexed by
coordinate vectors
$k \in {\mathbb N}^m,$
which have a partial ordering
given by
\begin{align}
 k \le n \iff k_j \le n_j \ (j \in \{1, \dots, m\}) \qquad (k, n \in {\mathbb N}^m).
\nonumber\end{align}
The root element of the tensor (corresponding to the top left entry of a
matrix) has index
$1_m = (1, 1, \dots, 1) \in {\mathbb N}^m.$
We shall also use the standard unit vectors
$e_j \in {\mathbb N}^m$
$(j \in \{1, \dots,m\}),$
where
$(e_j)_l = \delta_{j l}$
$(j, l \in \{1, \dots, m\}),$
i.e.
$e_j$
has $j$th entry 1 and all other entries 0.
\par\medskip
{\it Definition.}
Let
$m \in {\mathbb N}$
and
$n \in {\mathbb N}^m.$
Then
$M \in {\mathbb N}_0^n$
is called an
{\it order\/}
$m$
{\it tensor\/}
(of dimensions $n_1, n_2, \dots, n_m$).
It has entries
$M_k = M_{k_1, k_2, \dots, k_m} \in {\mathbb N}_0$
$(k \in {\mathbb N}^m, k \le n).$
\par
For
$j < m,$
we call any subtensor where
$m-j$
indices are fixed while the remaining
$j$
indices vary in the range determined by
$n$
an
{\it order $j$ slice\/}
of
$M.$
\par\medskip
{\it Remark.\/}
Strictly speaking, the order of the tensor is
$|\{j \in \{1, \dots, m\} : n_j > 1\}| \le m,$
so it has order
{\it at most\/}
$m.$
The order will be exactly
$m$
if
$n \in ({\mathbb N}+1)^m.$
However, we allow
$n \in {\mathbb N}^m$
for ease of reference later.
\par\medskip\noindent
The following is an extension of the vertex-cross sum property (V) of
matrices which states that the two pairs of diagonally opposite corners of any
rectangular submatrix add up to the same number \cite{rSupAlg}.
\par\medskip
{\it Definition.}
Let
$M \in {\mathbb N}_0^n,$
$n \in {\mathbb N}^m,$
$m \in {\mathbb N} + 1.$
Then we say that
$M$
has the
{\it vertex cross sum property\/}
(V) if and only if every order 2 slice of M has the property (V) for
matrices, i.e. if
\begin{align}
 M_{k_1, \dots, k_i, \dots, k_j, \dots, k_m} + M_{k_1, \dots, k_i', \dots, k_j', \dots, k_m} &= M_{k_1, \dots, k_i, \dots, k_j', \dots, k_m} + M_{k_1, \dots, k_i', \dots, k_j, \dots, k_m}
\label{eVdef}\end{align}
for all
$1\le i < j \le m$
and
$k_1, \dots, k_m, k_i', k_j' \in {\mathbb N}$
such that
$k_l, k_l' \le n_l$
$(l \in \{1, \dots, m\}).$
\begin{lemma}\label{lVprop}
Let
$M \in {\mathbb N}_0^n,$
$n \in {\mathbb N}^m,$
$m \in {\mathbb N}+1.$
Then
$M$
has property
{\rm (V)}
if and only if
\begin{align}
 M_k &= \sum_{j=1}^m M_{1_m + (k_j-1)e_j} - (m-1) M_{1_m} \qquad (k \in {\mathbb N}^m, k \le n).
\label{eVprop}\end{align}
\end{lemma}
\begin{proof}
Suppose
$M$
has property (V). We shall show by induction on
$l \in \{2, \dots, m\}$
that for any
$k \in {\mathbb N}^m,$
$k \le n,$
and any cardinality
$l$
subset
$\{j_1, j_2, \dots, j_l\} \subset \{1, 2, \dots, m\},$
\begin{align}
 M_{1_m + \sum_{r=1}^l (k_{j_r}-1)e_{j_r}} &= \sum_{r=1}^l M_{1_m + (k_{j_r}-1)e_{j_r}} - (l-1) M_{1_m}.
\label{eVind}\end{align}
For
$l = 2,$
property (V) gives
\begin{align}
 M_{1_m + (k_{j_1}-1)e_{j_1} + (k_{j_2}-1)e_{j_2}} &=  M_{1_m + (k_{j_1}-1)e_{j_1}} + M_{1_m + (k_{j_2}-1)e_{j_2}} -  M_{1_m}.
\nonumber\end{align}
Now suppose
$l \in \{2, \dots, m-1\}$
is such that identity (\ref{eVind}) holds for up to
$l$
terms. Then, again by property (V), we find
\begin{align}
 M_{1_m + \sum_{r=1}^{l+1} (k_{j_r}-1)e_{j_r}} &= M_{1_m + \sum_{r=1}^{l-1} (k_{j_r}-1)e_{j_r} + (k_{j_l}-1)e_{j_l}} + M_{1_m + \sum_{r=1}^{l-1} (k_{j_r}-1)e_{j_r} + (k_{j_{l+1}}-1)e_{j_{l+1}}}
\nonumber\\
 &\qquad  - M_{1_m + \sum_{r=1}^{l-1} (k_{j_r}-1)e_{j_r}}
\nonumber\\
 &= \sum_{r=1}^l M_{1_m + (k_{j_r}-1)e_{j_r}} + \sum_{r \in \{1, \dots,l-1\} \cup \{l+1\}} M_{1_m + (k_{j_r}-1)e_{j_r}}
\nonumber\\
 &\qquad  - 2\,(l-1) M_{1_m}
 - \sum_{r=1}^{l-1} M_{1_m + (k_{j_r}-1)e_{j_r}} + (l-2) M_{1_m}
\nonumber\\
 &= \sum_{r=1}^{l+1} M_{1_m + (k_{j_r}-1)e_{j_r}} - l\,M_{1_m}.
\nonumber\end{align}
The identity (\ref{eVprop}) now follows when we take
$l = m$
in (\ref{eVind}), which forces
$\{j_1, j_2, \dots, j_m\} = \{1, 2, \dots, m\}.$
The converse follows directly from applying identity (\ref{eVprop}) to (\ref{eVdef}).
\end{proof}
The preceding lemma shows that if the root entry
$M_{1_m} = 0,$
then each entry of the order
$m$
tensor
$M$
is the sum of the entries on the axes for each of its index coordinates, i.e.
\begin{align}
 M_k &= M_{k_1,1,\dots,1} + M_{1, k_2, 1, \dots, 1} + \cdots + M_{1, \dots, 1, k_m}.
\nonumber\end{align}
This means that overall the set of entries of
$M$
is equal to the sum set of the sets of entries on each coordinate axis,
where all but one entry of the index vector are kept equal to 1.
This gives the following connection with sum systems.
\begin{theorem}\label{tssten}
Let
$m \in {\mathbb N},$
$n \in ({\mathbb N}+1)^m$
and
$M \in {\mathbb N}_0^n$
an order
$m$
tensor with property
{\rm (V),}
$M_{1_m} = 0$
and set of entries
\begin{align}
 \{M_k : k \in {\mathbb N}^m, k \le n\} &= \Ap{\prod_{j=1}^m n_j}.
\nonumber\end{align}
Then the sets
$A_1, A_2, \dots, A_m \subset {\mathbb N}_0,$
\begin{align}
 A_j &= \{M_{1_m + k e_j} : k \in \ap{n_j}\} \qquad (j \in \{1, \dots, m\})
\label{eMcoord}\end{align}
form an $m$-part sum system.
\end{theorem}
\begin{proof}
The statement follows from Lemma \ref{lVprop} when we note that the identity
(\ref{eVprop}) will turn into
\begin{align}
 M_k &= \sum_{j=1}^m M_{1_m + (k_j-1)e_j} \qquad (k \in {\mathbb N}^m, k \le n),
\label{eMsum}\end{align}
and that the set of entries of
$M$
is equal to the target set for the sum system.
\end{proof}
Conversely, given an $m$-part sum system and choosing the entries on the
coordinate axes of
$M$
such that they satisfy (\ref{eMcoord}) and
$M_{1_m} = 0,$
it is clear that defining the remaining entries via (\ref{eMsum}) will result
in an order
$m$
tensor with property (V).
\par
In fact,
$M$
can be considered as an $m$-dimensional tabular representation of the
sum system with a certain arrangement of the elements of each component set.
\par
There is some freedom of choice in assigning the elements of the component
sets of a sum system to tensor entries so as to satisfy (\ref{eMcoord}), with
only the constraint that
$M_{1_m} = 0.$
In order to establish a bijection, we introduce the following generalisation
of Ollerenshaw and Br\'ee's definition of a principal reversible square \cite{rOB}.
\par\medskip
{\it Definition.}
We call an order $m$ tensor
$M \in {\mathbb N}_0^n,$
$n \in ({\mathbb N}+1)^m,$
$m \in N,$
a
{\it principal reversible $m$-cuboid\/}
if
$M$
has property (V), its set of entries is
\begin{align}
 \{M_k : k \in {\mathbb N}^m, k \le n\} &= \Ap{\prod_{j=1}^m n_j},
\nonumber\end{align}
and for every
$j \in \{1, \dots, m\},$
every row in the $j$th direction is arranged in strictly increasing order,
i.e.
$M_{k} < M_{k + l e_j} \ (k \in {\mathbb N}^m, 1 \le l \le n_j-k_j).$
\par\medskip\noindent
Putting the elements of the sum system component
$A_j$
onto the
$j$th coordinate axis of
$M,$
we obtain the following relationship by virtue of Theorem \ref{tssten}.
\begin{corollary}\label{cssprc}
Let
$m \in {\mathbb N}.$
There is a bijection between the principal reversible $m$-cuboids with
dimension vector
$n \in ({\mathbb N}+1)^m$
and the $m$-part sum systems
$A_1, \dots, A_m$
with cardinalities
$|A_j| = n_j$
$(j \in \{1, \dots, m\}).$
\end{corollary}
\par\medskip\noindent
In conjunction with Theorem \ref{tsspalin}, this shows that principal reversible
$m$-tensors also have a generalised form of the row and column reversal symmetry (R)
defined for matrices \cite{rSupAlg}, as follows.
\begin{theorem}\label{trevsym}
Let
$m \in {\mathbb N},$
$n \in ({\mathbb N}+1)^m,$
and let
$M \in {\mathbb N}_0^n$
be a principal reversible $m$-cuboid. Then
$M$
has the line reversal symmetry
{\rm (R),}
i.e. for all
$j \in \{1, \dots, m\}$
and any
$k \in {\mathbb N}^m, k \le n,$
\begin{align}
 &M_{k_1, \dots, k_{j-1}, l, k_{j+1}, \dots, k_m} + M_{k_1, \dots, k_{j-1}, n_j+1-l, k_{j+1}, \dots, k_m}
\nonumber\\
 &= M_{k_1, \dots, k_{j-1}, 1, k_{j+1}, \dots, k_m} + M_{k_1, \dots, k_{j-1}, n_j, k_{j+1}, \dots, k_m}
 \qquad (l \in \{1, \dots, n_j\}).
\label{eRcub}\end{align}
\end{theorem}
\begin{proof}
Let
$A_1, A_2, \dots, A_m$
be the sum system corresponding to
$M$
by Theorem \ref{tssten}.
Then, for each
$j \in \{1, \dots, m\},$
\begin{align}
 0 &= M_{1_m} < M_{1_m + e_j} < M_{1_m + 2 e_j} < \cdots < M_{1_m + (n_j - 1) e_j} = \max A_j
\nonumber\end{align}
are the elements of the component set
$A_j,$
which by Theorem \ref{tsspalin} has the palindromic property
\begin{align}
 M_{1_m + k e_j} + M_{1_m + (n_j -1-k) e_j} &= M_{1_m + (n_j-1) e_j} + M_{1_m} \qquad (k \in \ap{n_j}).
\nonumber\end{align}
This proves the identity (\ref{eRcub}) along the coordinate axes of
$M;$
the general case follows by observing that Lemma \ref{lVprop} gives the representation
\begin{align}
 M_k &= M_{k_1, \dots, k_{j-1}, 1, k_{j+1}, \dots, k_m} + M_{1_m + (k_j-1) e_j}
 \qquad (j \in \{1, \dots, m\}, k \in {\mathbb N}^m, k \le n)
\nonumber\end{align}
for the entries of
$M.$
\end{proof}
\section{Structure and construction of principal reversible cuboids}
\label{sconstr}
Throughout this section, let
$m \in {\mathbb N},$
$n \in {\mathbb N}^m \setminus \{1_m\},$
and consider a principal reversible cuboid
$M \in {\mathbb N}_0^n.$
For a multiindex
$\tilde n \in {\mathbb N}^m,$
$\tilde n \le n,$
we write
\begin{align}
 M_{[\tilde n]} := (M_k)_{k \le \tilde n}
\nonumber\end{align}
for the subcuboid of
$M$
which has dimensions
$\tilde n$
and includes the root entry
$M_{1_m}.$
Moreover, we define
\begin{align}
 \mu_{\tilde n} := \min \{N \in {\mathbb N} : N \neq M_k \ (k \le \tilde n)\},
\nonumber\end{align}
the smallest positive integer not appearing as an entry in
$M_{[\tilde n]}.$
Then we have the following characterisation of principal reversible subcuboids
of
$M.$
\begin{lemma}\label{lPRsC}
For
$\tilde n \le n,$
\begin{align}
 \mu_{\tilde n} &\le \prod_{j=1}^m \tilde n_j,
\label{ePRsC}\end{align}
with equality
if and only if
$M_{[\tilde n]}$
is a principal reversible cuboid.
\end{lemma}
\begin{proof}
$M_{[\tilde n]}$
inherits the ordering property and (V) from
$M.$
Hence it is a principal reversible cuboid if and only if
\begin{align}
 \{M_k : k \le \tilde n\} &= \Ap{\prod_{j=1}^m \tilde n_j}.
\nonumber\end{align}
If this is the case, then evidently (\ref{ePRsC}) holds;
if it is not the case, then
$M_{[\tilde n]},$
having
$\prod_{j=1}^m \tilde n_j$
different entries, must skip some element of
$\Ap{\prod_{j=1}^m \tilde n_j},$
so
$\mu_{\tilde n} < \prod_{j=1}^m \tilde n_j.$
\end{proof}
By Lemma \ref{lVprop}, the entries of
$M$
arise as sums of the corresponding entries along the coordinate axes of
$M.$
Let us define
\begin{align}
 a_{j,k} := M_{1_m + k e_j} &\qquad (j \in \{1, \dots, m\}, k \in \ap{n_j});
\nonumber\end{align}
then
the identity (\ref{eVprop}), with
$M_{1_m} = 0,$
gives
\begin{align}
 M_{1_m + k} &= \sum_{j=1}^m M_{1_m + k_j e_j} = \sum_{j=1}^m a_{j, k_j}
 \qquad (k \in {\mathbb N}_0^m, 1_m + k \le n).
\label{eMina}\end{align}
The following observation shows that for any subcuboid (containing the root element)
$M_{[\tilde n]},$
the smallest missing integer
$\mu_{\tilde n}$
appears on a coordinate axis of
$M,$
just outside
$M_{[\tilde n]}.$
\begin{lemma}\label{lfindmu}
Let
$\tilde n \le n,$
$\tilde n \neq n.$
Then there is
$j \in \{1, \dots, m\}$
such that
$\mu_{\tilde n} = a_{j, \tilde n_j}.$
\end{lemma}
\begin{proof}
By definition,
$\mu_{\tilde n}$
is the smallest entry of
$M$
outside
$M_{[\tilde n]},$
\begin{align}
 \mu_{\tilde n} &= \min\{M_{\hat n} : \hat n \le n, \hat n \not \le \tilde n\}.
\nonumber\end{align}
By the increasing arrangement of all lines parallel to coordinate axes, we have for any
$\hat n$
and any
$j \in \{1, \dots, m\}$
that
$a_{j, \hat n_j -1} = M_{1_m + (\hat n_j -1)e_j} \le M_{\hat n},$
so
\begin{align}
 \mu_{\tilde n} &= \min \{a_{j, \hat n_j -1} : j \in \{1, \dots, m\}, \tilde n_j < \hat n_j < n_j\}
\nonumber\\
 &= \min \{a_{j, \tilde n_j} : j \in \{1, \dots, m\}\},
\nonumber\end{align}
by the increasing arrangement of
$a_{j, \cdot}.$
\end{proof}
If the cuboid in Lemma \ref{lfindmu} arises by truncating a principal reversible
subcuboid in one direction only, the smallest missing integer must appear on
the axis of the direction of truncation, since the other directions would lead
outside the larger enclosing principal reversible subcuboid.
\begin{corollary}\label{ccutmu}
Let
$\tilde n \in {\mathbb N}^m,$
$\tilde n \le n,$
such that
$M_{[\tilde n]}$
is a principal reversible subcuboid. Let
$j \in \{1, \dots, m\}$
and
$\hat n \in {\mathbb N}^m$
such that
$\hat n_i = \tilde n_i$
$(i \neq j)$
and
$\hat n_j < \tilde n_j.$
Then
$\mu_{\hat n} = a_{j, \hat n_j}.$
\end{corollary}
The following statement gives an extension of this beyond the confines of the
principal reversible subcuboid.
\begin{lemma}\label{lmuchain}
Let
$\tilde n \le n,$
$\tilde n \neq n,$
be such that
$M_{[\tilde n]}$
is a proper principal reversible subcuboid of
$M,$
and let
$j \in \{1, \dots, m\}.$
Suppose for some
$k_0 \in {\mathbb N},$
$k_0 < \tilde n_j,$
\begin{align}
 a_{j, \tilde n_j + k} &= a_{j, k} + \prod_{i=1}^m \tilde n_i \qquad (k \in \ap{k_0}).
\nonumber\end{align}
Then
$\mu_{\tilde n + k_0 e_j} = a_{j, k_0} + \prod_{i=1}^m \tilde n_i.$
\end{lemma}
\begin{proof}
By definition,
$\mu_{\tilde n + k_0 e_j}$
is the smallest positive integer not in the set
$\{M_{\hat n} : \hat n \le \tilde n + k_0 e_j\}.$
Since, by Lemma \ref{lPRsC},
$\{M_{\hat n} : \hat n \le \tilde n\} = \Ap{\prod_{i=1}^m \tilde n_i},$
we have in fact that
$\mu_{\tilde n + k_0 e_j}       $
is the smallest positive integer not in the set
\begin{align}
 \{M_{\hat n} &: \hat n_i \le \tilde n_i \ (i \neq j), \tilde n_j < \hat n_j \le \tilde n_j + k_0\}
\nonumber\\
 &= \left \{a_{j, \tilde n_j + k} + \sum_{i \neq j} a_{i, k_i} : k_i \in \ap{\tilde n_i} \ (i \neq j), k \in \ap{k_0} \right\}
\nonumber\\
 &= \left\{ \prod_{i=1}^m \tilde n_i + \sum_{i=1}^m a_{i, k_i} : k_i \in \ap{\tilde n_i} \ (i \neq j), k_j \in \ap{k_0} \right\}
\nonumber\\
 &= \prod_{i=1}^m \tilde n_i + \{M_{\hat n} : \hat n_i \le \tilde n_i \ (i \neq j), \hat n_j \le k_0\},
\nonumber\end{align}
using (\ref{eMina}) in the first and the hypothesis of the Lemma in the second
equality.
Taking the minimum on both sides, we find
$\mu_{\tilde n + k_0 e_j} = \mu_{\tilde n + (k_0 - \tilde n_j) e_j} + \prod_{i=1}^m \tilde n_i.$
Corollary \ref{ccutmu} gives
$\mu_{\tilde n + (k_0 - \tilde n_j) e_j} = a_{j, k_0},$
and the statement follows.
\end{proof}
The following lemma provides the key to understanding the structure of
principal reversible cuboids. Essentially it shows that, starting from
a principal reversible subcuboid, finding the entry of
$M$
giving the next integer in sequence and adding the slice in the
corresponding direction to the subcuboid, and continuing in this way,
the next integer in sequence will always be found in the same direction as
the previous one, until the addition of slices has completed a larger
principal reversible subcuboid (or exhausted $M$).
Thus the next integer in sequence can only appear in a new direction if
the starting point is a complete principal reversible subcuboid, not a
general subcuboid.
\begin{lemma}\label{lMkey}
Suppose
$M_{[\tilde n]}$
is a proper principal reversible subcuboid of
$M,$
$\tilde n \le n,$
$\tilde n \neq n,$
such that
for some
$j \in \{1, \dots, m\}$
and some
$K \in \{1, \dots, \tilde n_j-1\},$
\begin{align}
 a_{j, \tilde n_j + k} &= \mu_{\tilde n + k e_j} \qquad (k \in \ap{K}).
\nonumber\end{align}
Then either
$M_{[\tilde n + K e_j]}$
is a principal reversible cuboid, or
$\mu_{\tilde n + K e_j} = a_{j, \tilde n_j + K}.$
\end{lemma}
\begin{proof}
Applying Lemma \ref{lmuchain} recursively to
$k_0 \in \{1, \dots, K\},$
we find that
\begin{align}
 a_{j, \tilde n_j + K -1} &= \mu_{\tilde n + (K-1) e_j} = a_{j, K-1} + \prod_{i=1}^m \tilde n_i,
\nonumber\end{align}
and
\begin{align}
 \mu_{\tilde n + K e_j} &= a_{j,K} + \prod_{i=1}^m \tilde n_i.
\label{e2}\end{align}
Hence the slice of
$M$
with indices
$\hat n_i \le \tilde n_i$
$(i \neq j),$
$\hat n_j = \tilde n_j + K$
has entries
\begin{align}
 S_1 &:= \{M_{\hat n} : \hat n_i \le \tilde n_i \ (i \neq j), \hat n_j = \tilde n_j + K\}
\nonumber\\
 &= \left\{ \sum_{i=1}^m a_{i,k_i} : 0 \le k_i \le \tilde n - 1 \ (i \neq j), k_j = \tilde n_j + K - 1 \right\}
\nonumber\\
 &= \left\{ a_{j,K-1} + \prod_{r=1}^m \tilde n_r + \sum_{i \neq j} a_{i,k_i} : 0 \le k_i \le \tilde n - 1 \ (i \neq j) \right\}
\nonumber\end{align}
using (\ref{eMina}).
Now suppose that
$\mu_{\tilde n + K e_j} \neq a_{j, \tilde n_j + K}.$
By Lemma \ref{lfindmu}, there is then
$l~\in~\{1, \dots, m\},$
$l \neq j,$
such that
$\mu_{\tilde n + K e_j} = a_{l, \tilde n_l}.$
This means (again by (\ref{eMina})) that the slice of
$M$
with indices
$\hat n_i \le \tilde n_i$
$(i \neq l),$
$\hat n_l = \tilde n_l + 1$
has entries
\begin{align}
 S_2 &:= \{M_{\hat n} : \hat n_i \le \tilde n_i \ (i \neq l), \hat n_l = \tilde n_l + 1\}
\nonumber\\
 &= \left\{\sum_{i=1}^m a_{i, \tilde k_i} : 0 \le \tilde k_i \le \tilde n_i - 1 \ (i \neq l), \tilde k_l = \tilde n_l \right\}.
\nonumber\end{align}
As the index sets appearing in the definitions of
$S_1$
and
$S_2$
are disjoint and all entries of
$M$
are different, it follows that
$S_1 \cap S_2 = \emptyset.$
\par
Now if
$M_{[\tilde n + (K- \tilde n_j) e_j]}$
is a principal reversible subcuboid, then
$M_{[\tilde n + K e_j]}$
will have entries
\begin{align}
 \Ap{\prod_{i=1}^m \tilde n_i} \cup \left(\prod_{i=1}^m \tilde n_i + \Ap{K \prod_{i \neq j} \tilde n_i}\right) &= \Ap{(\tilde n_j + K) \prod_{i \neq j}\tilde n_i}
\nonumber\end{align}
and hence be a principal reversible subcuboid.
If, on the other hand,
$M_{[\tilde n + (K- \tilde n_j) e_j]}$
is not a principal reversible subcuboid, then by Lemma \ref{lPRsC} and Corollary \ref{ccutmu},
\begin{align}
 a_{j,K} &= \mu_{\tilde n + (K- \tilde n_j) e_j} < K \prod_{i \neq j} \tilde n_i;
\nonumber\end{align}
also,
\begin{align}
 a_{j, \tilde n_j - K} &= \mu_{\tilde n - K e_j} \le (\tilde n_j-K) \prod_{i \neq j} \tilde n_i,
\nonumber\end{align}
so
$\displaystyle a_{j,K} + a_{j, \tilde n_j - K} < \prod_{i=1}^m \tilde n_i.$
As
$M_{[\tilde n]}$
is a principal reversible subcuboid with entries
$\displaystyle \Ap{\prod_{i=1}^m \tilde n_i},$
there are suitable indices
$k_i \in \ap{\tilde n_i}$
$(i \in \{1, \dots, m\})$
such that
\begin{align}
 a_{j,K} + a_{j, \tilde n_j - K} &= \sum_{i=1}^m a_{i, k_i}.
\nonumber\end{align}
Setting
$\tilde k_j := \tilde n_j - 1 - k_j \in \ap{\tilde n_j},$
this equation can be written in the form
\begin{align}
 a_{j,K} + a_{j,\tilde n_j-K} &= a_{j,\tilde n_j - 1 - \tilde k_j} + \sum_{i \neq j} a_{i,k_i}.
\label{e1}\end{align}
By the reversal symmetry of the principal reversible subcuboid
$M_{[\tilde n]}$
(Theorem \ref{trevsym}),
the numbers
$0 = a_{j,0} < a_{j,1} < \cdots < a_{j,\tilde n_j-1}$
have the property
\begin{align}
 a_{j, \tilde n_j -1} &= a_{j,r} - a_{j, \tilde n_j-1-r} \qquad (r \in \ap{\tilde n_j}),
\nonumber\end{align}
so in particular
$a_{j, \tilde n_j-1-\tilde k_j} = a_{j, \tilde n_j-1} - a_{j, \tilde k_j}$
and
$a_{j,\tilde n_j - K} = a_{j, \tilde n_j-1} - a_{j, K-1}.$
Hence equation (\ref{e1}) is equivalent to
\begin{align}
 a_{j,K} + a_{j, \tilde k_j} &= a_{j,K-1} + \sum_{i \neq j} a_{i,k_i}.
\nonumber\end{align}
Taking into account (\ref{e2}), we hence find
\begin{align}
 S_2 \ni a_{l, \tilde n_l} + a_{j, \tilde k_j} &= \mu_{\tilde n + K e_j} + a_{j, \tilde k_j}
 = a_{j,K-1} + \prod_{r=1}^m \tilde n_r + \sum_{i \neq j} a_{i,k_i} \in S_1.
\nonumber\end{align}
This contradicts the fact that
$S_1$
and
$S_2$
are disjoint.
\end{proof}
Clearly, given two principal reversible subcuboids of
$M,$
one must contain the other, since both contain a consecutive sequence of
integers starting from 0 and the entries of
$M$
are all different.
Therefore the concept of
{\it maximality\/}
of a proper principal reversible subcuboid of
$M$
is well-defined, and there is a unique maximal proper principal reversible subcuboid of
$M.$
\begin{theorem}\label{tbasicstruc}
Let
$m \in {\mathbb N}+1,$
$n \in {\mathbb N}^m \setminus \{1_m\},$
and
$M \in {\mathbb N}_0^n$
a principal reversible cuboid.
If
$M_{[\tilde n]},$
with
$\tilde n \le n,$
$\tilde n \neq n,$
is a maximal proper principal reversible subcuboid, then there is
$j \in \{1, \dots, m\}$
such that
$\tilde n_i = n_i$
$(i \neq j).$
\end{theorem}
\begin{proof}
Suppose
$\tilde n_j < n_j$
and
$\tilde n_k < n_k$
for some
$j \neq k.$
Then, by Lemma \ref{lfindmu},
$\mu_{\tilde n} \in \{a_{j,\tilde n_j}, a_{k, \tilde n_k}\};$
w.l.o.g. let
$\mu_{\tilde n} = a_{j,\tilde n_j}.$
Then, by Lemma  \ref{lMkey},
the next smallest missing number from each extension
$M_{[\tilde n + e_j]},$
$M_{[\tilde n + 2 e_j]}, \dots$
of
$M_{[\tilde n]}$
in direction
$j$
will again be found in direction
$j,$
until a larger principal reversible subcuboid
$M_{[\tilde n + K e_j]}$
is completed, with some
$K > 0.$
As the $k$th entry of the multiindex
$\tilde n + K e_j$
is equal to
$\tilde n_k < n_k,$
$M_{[\tilde n + K e_j]} \neq M;$
on the other hand,
$M_{[\tilde n]}$
is a proper principal reversible subcuboid of
$M_{[\tilde n + K e_j]},$
contradicting its maximality.
\end{proof}
\begin{theorem}\label{tbasicfact}
Let
$m \in {\mathbb N}+1,$
$n \in {\mathbb N}^m \setminus \{1_m\}$
and
$M \in {\mathbb N}_0^n$
a principal reversible cuboid.
Then there is some
$j \in \{1, \dots, m\}$
and some
$\tilde n \in {\mathbb N}^m$
such that
$\tilde n_i = n_i$
$(i \neq j),$
$\tilde n_j < n_j,$
$\tilde n_j | n_j$
and
$M_{[\tilde n]}$
is a principal reversible subcuboid of
$M.$
Moreover, for any
$\hat n \le \tilde n$
and
$k \in \Ap{\frac{n_j}{\tilde n_j}},$
we have
$M_{\hat n + k \tilde n_j e_j} = M_{\hat n} + k \prod_{i=1}^m \tilde n_i.$
\end{theorem}
\begin{proof}
Since
$M$
is not just the trivial $m$-cuboid
$(0) \in {\mathbb N}_0^{1_m},$
it has a maximal proper principal reversible subcuboid
$M_{n'},$
with
$n' \le n,$
$n' \neq n;$
note that
$n' = 1_m$
and hence
$M_{n'} = (0)$
is possible.
By Theorem \ref{tbasicstruc}, there is
$j \in \{1, \dots, m\}$
such that
$n'_i = n_i$
$(i \neq j)$
and
$n'_j < n_j.$
Let
$\tilde n \in {\mathbb N}^m$
be such that
$\tilde n_i = n_i$
$(i \neq j)$
and
such that
$\tilde n_j$
is the minimal number for which
$M_{[\tilde n]}$
is a principal reversible subcuboid of
$M.$
\par
By Lemmas \ref{lmuchain} and \ref{lMkey} and eq. (\ref{eMina}), we find
\begin{align}
 M_{\hat n + \tilde n_j e_j} &= \sum_{i \neq j} a_{i, \hat n_i - 1} + a_{j, \hat n_j - 1} + \prod_{r=1}^m \tilde n_r
 = M_{\hat n} + \prod_{r=1}^m \tilde n_r
\nonumber\end{align}
for
$\hat n \le \tilde n,$
and this makes
$M_{[\tilde n + \tilde n_j e_j]}$
a principal reversible subcuboid which is composed of
$M_{[\tilde n]}$
and an a copy of this cuboid with entries offset with
$\prod_{r=1}^m \tilde n_r.$
Applying the same reasoning to this larger subcuboid, if it is not already
equal to
$M,$
gives the last statement of the Theorem.
\par
By minimality of
$M_{[\tilde n]},$
the principal reversible cuboid
$M$
must be composed of a number of complete offset copies of it to contain a
complete arithmetic sequence. Hence it follows that
$\tilde n_j$
is a divisor of
$n_j.$
\end{proof}
\section{Building operators and joint ordered factorisations}\label{sjof}
Theorem \ref{tbasicfact} has shown that every principal reversible cuboid, except
the trivial
$(0) \in {\mathbb N}_0^{1_m},$
is composed of shifted copies of a smaller principal reversible cuboid, stacked
in one of the
$m$
directions. By recursion, this observation gives a description of principal
reversible cuboids which can be used to construct them. In order to make the
construction more transparent, we introduce building operators which describe
the stacking process.
\par
We shall use the following notation.
For
$k \in {\mathbb N},$
we denote the arithmetic progression vector by
\def\apv#1{\overrightarrow{\ap{#1}}}
\def\Apv#1{\overrightarrow{\Ap{#1}}}
\begin{align}
 \apv k &= (0, 1, 2, \dots, k-1).
\nonumber\end{align}
Moreover, for
$m \in {\mathbb N}$
and any multiindex
$n \in {\mathbb N}^m,$
we write
$1_{[n]}$
for the cuboid with dimension vector
$n$
and all entries equal to 1.
\par\medskip
{\it Definition.}
Let
$k, m \in {\mathbb N},$
$j \in \{1, \dots, m\},$
$n \in {\mathbb N}^m,$
$v \in {\mathbb N}_0^k$
and
$M \in {\mathbb N}_0^n.$
Then we define the
{\it direction $j$ Kronecker product\/}
of
$v$
with
$M$
as
$v \otimes_j M,$
where
\begin{align}
 (v \otimes_j M)_{\hat n + l n_j e_j} &= v_{l+1} M_{\hat n}
 \qquad (\hat n \le n, l \in \ap k).
\nonumber\end{align}
If
$m = 1,$
then this product turns into the standard Kronecker product of the vectors
$v \in {\mathbb N}_0^k$
and
$w = M \in {\mathbb N}_0^{n_1},$
i.e.
\begin{align}
 (v \otimes w)_{l_1 n_1 + l_2 + 1} &= v_{l_1+1} w_{l_2+1} \qquad (l_1 \in \ap k, l_2 \in \ap{n_1}).
\nonumber\end{align}
This product is obviously bilinear.
\begin{lemma}\label{lprass}
The direction $j$ Kronecker product is associative, i.e. for
$k_1, k_2, m \in {\mathbb N},$
$n \in {\mathbb N}^m,$
$v \in {\mathbb N}_0^{k_1},$
$w \in {\mathbb N}_0^{k_2}$
and
$M \in {\mathbb N}_0^n,$
\begin{align}
 v \otimes_j (w \otimes_j M) &= (v \otimes w) \otimes_j M.
\nonumber\end{align}
\end{lemma}
\begin{proof}
For any
$\hat n \le n$
and
$l = l_1 k_2 + l_2 \in \ap{k_1 k_2},$
$l_1 \in \ap {k_1},$
$l_2 \in \ap {k_2},$
we find
\begin{align}
 (v \otimes_j (w \otimes_j M))_{\hat n + l n_j e_j} &= (v \otimes_j (w \otimes_j M))_{\hat n + l_2 n_j e_j + l_1 k_2 n_j e_j}
\nonumber\\
 &= v_{l_1+1} (w \otimes_j M)_{\hat n + l_1 n_j e_j}
 = v_{l_1+1} w_{l_2+1} M_{\hat n}
\nonumber\\
 &= (v \otimes w)_{l+1} M_{\hat n}
 = ((v \otimes w) \otimes_j M)_{\hat n + l n_j e_j},
\nonumber\end{align}
as required.
\end{proof}
{\it Definition.}
Let
$k, m \in {\mathbb N},$
$j \in \{1, \dots, m\}.$
Then we define the
{\it building operator\/}
\def\cbo{{\cal B}}
$\cbo_{j,k}$
as the operation which turns any cuboid
$M \in {\mathbb N}_0^n,$
$n \in {\mathbb N}^m,$
into
\begin{align}
 \cbo_{j,k}(M) &= \left(\prod_{r=1}^m n_r \right) \apv k \otimes_j 1_{[n]} + 1_k \otimes_j M \in {\mathbb N}_0^{n + (k-1)n_j e_j}.
\nonumber\end{align}
\par\medskip\noindent
The following observation shows that the composition of two building
operators working in the same coordinate direction is just one
building operator.
\begin{lemma}\label{lcborepeat}
Let
$k_1, k_2, m \in {\mathbb N}$
and
$j \in \{1, \dots, m\}.$
Then
$\cbo_{j,k_1} \circ \cbo_{j,k_2} = \cbo_{j,k_1 k_2}.$
\end{lemma}
\begin{proof}
Let
$M \in {\mathbb N}_0^n,$
$n \in {\mathbb N}^m.$
Then
\begin{align}
 \cbo_{j,k_1} \circ \cbo_{j,k_2} (M) &= \cbo_{j,k_1} \left(\left(\prod_{r=1}^m n_r \right) \apv{k_2} \otimes_j 1_{[n]} + 1_{k_2} \otimes_j M \right)
\nonumber\\
 &= \left(\prod_{r=1}^m n_r \right) k_2\,\apv{k_1} \otimes_j 1_{[n + (k_2-1)n_j e_j]}
\nonumber\\
 &\qquad
 + 1_{k_1} \otimes_j \left(\left(\prod_{r=1}^m n_r \right) \apv{k_2} \otimes_j 1_{[n]} + 1_{k_2} \otimes_j M \right)
\nonumber\\
 &= \left(\prod_{r=1}^m n_r \right) (k_2\,\apv{k_1} \otimes 1_{k_2} + 1_{k_1} \otimes \apv{k_2}) \otimes_j 1_{[n]} + (1_{k_1} \otimes 1_{k_2}) \otimes_j M
\nonumber\\
 &= \left(\prod_{r=1}^m n_r \right) \apv{k_1 k_2} \otimes_j 1_{[n]} + 1_{k_1 k_2} \otimes_j M
 = \cbo_{j,k_1 k_2} (M),
\nonumber\end{align}
using Lemma \ref{lprass} in the penultimate line.
\end{proof}
Applying this setup in conjunction with Theorem \ref{tbasicfact}, we can deduce the following structure theorem for
principal reversible cuboids.
\begin{theorem}\label{tcbuild}
Let
$m \in {\mathbb N},$
$n \in {\mathbb N}^m$
and
$M \in {\mathbb N}_0^n$
a principal reversible cuboid.
Then there is a number
$L \in {\mathbb N}$
and numbers
$j_l \in \{1, \dots, m\},$
$f_l \in {\mathbb N}+1$
$(l \in \{1, \dots, L\})$
such that
$\prod_{j_l = j} f_l = n_j$
$(j \in \{1, \dots, m\})$
and
\begin{align}
 M &= \cbo_{j_L, f_L} \circ \cbo_{j_{L-1}, f_{L-1}} \circ \cdots \circ \cbo_{j_1, f_1} ((0)),
\label{ecbuild}\end{align}
where
$(0) \in {\mathbb N}_0^{1_m}$
is the trivial principal reversible cuboid.
\hfill\break
Without loss of generality, we can assume that
$j_l \neq j_{l-1}$
$(l \in \{2, \dots, L\}).$
\end{theorem}
\begin{proof}
Using the building operator defined above, the statement of Theorem \ref{tbasicfact}
can be paraphrased in the following way. There is some
$j \in \{1, \dots, m\}$
and
$f \in {\mathbb N},$
$f | n_j,$
such that
$M = \cbo_{j,f} (M_{[\tilde n]}),$
where
$\tilde n \in {\mathbb N}^m$
has entries
$\tilde n_i = n_i$
$(i \neq j)$
and
$\tilde n_j = n_j / f.$
Here
$M_{[\tilde n]}$
is again a principal reversible cuboid. Unless this is the trivial cuboid
$(0) \in {\mathbb N}_0^{1_m},$
we can again apply Theorem \ref{tbasicfact} to it, and thus recursively obtain
the building operator chain in (\ref{ecbuild}).
The last statement reflects Lemma \ref{lcborepeat},
which allows fusion of consecutive building operators in the same direction
into one.
\end{proof}
Theorem \ref{tcbuild} shows that principal reversible cuboids are obtained from
building operator chains; the coefficients of such a chain arise from
factorising the individual dimensions
$n_j$
$(j \in \{1, \dots, m\})$
of the principal reversible cuboid, and arranging the factors in a sequence
such that
consecutive factors in the sequence belong to different
coordinate directions.
\par
Note that in the special (and untypical) case
$m = 2,$
this condition (which corresponds to the last sentence in Theorem \ref{tcbuild})
enforces alternation of directions
$j_1 = 1,$
$j_2 = 2,$
$j_3 = 1,$
$j_4 = 2,$
etc.\ (or the analogue starting with $j_2 = 2$), ending with either the same or
the other direction depending on whether
$L$
is odd or even.
If
$n_1 = n_2$
and we start with
$j_1 = 1,$
this gives a building operator chain equivalent to Ollerenshaw and Br\'ee's
construction of principal reversible squares \cite{rOB}.
However, if
$m > 2,$
then the possible patterns are considerably more complex.
\par\medskip
{\it Definition.}
Let
$m \in {\mathbb N}$
and
$n \in {\mathbb N}^m.$
Then we call
\begin{align}
 ((j_1, f_1), (j_2, f_2), \dots, (j_L, f_L)) \in (\{1, \dots, m\} \times ({\mathbb N}+1))^L,
\nonumber\end{align}
where
$L \in {\mathbb N},$
a
{\it joint ordered factorisation\/}
of
$n = (n_1, \dots, n_m)$
if
\begin{align}
 \prod_{j_l = j} f_l &= n_j \qquad (j \in \{1, \dots, m\})
\nonumber\end{align}
and
$j_l \neq j_{l-1}$
$(l \in \{2, \dots, L\}).$
\par\medskip\noindent
By Theorem \ref{tssten}, the entries on the coordinate axes of a principal
reversible cuboid form a sum system (with the entries of each component set
appearing in increasing order on the corresponding axis, and the coordinate
axes arranged in the order of the smallest non-zero entry of the component
sets).
Thus the building operator chain of Theorem \ref{tcbuild} also gives rise to a
construction for the corresponding sum system, as follows.
\begin{theorem}\label{tssbuild}
Let
$m \in {\mathbb N}.$
Suppose the sets
$A_1, A_2, \dots, A_m \subset {\mathbb N}_0$
form a sum system.
Let
$n_j := |A_j|$
$(j \in \{1, \dots, m\}).$

Then there is a joint ordered factorisation
$((j_1, f_1), \dots, (j_L, f_L))$
of
$(n_1, \dots, n_m)$
such that
\begin{align}
 A_j &= \sum_{j_l = j} \left(\prod_{s=1}^{l-1} f_s \right) \ap{f_l}
 \qquad (j \in \{1, \dots, m\}).
\label{essbuild}\end{align}
Conversely, given any joint ordered factorisation of
$n \in {\mathbb N}^m,$
$(\ref{essbuild})$
generates an $m$-part sum system.
\end{theorem}
\begin{proof}
By Corollary \ref{cssprc}, there is a one-to-one relationship between $m$-part
sum systems and principal reversible $m$-cuboids with dimension vector
$n \in ({\mathbb N}+1)^m.$
Consider a principal reversible cuboid
$M \in {\mathbb N}_0^n$
and its corresponding building operator chain (\ref{ecbuild}).
Set
$M^{(0)} = (0) \in {\mathbb N}_0^{1_m}$
and
$M^{(l)} = \cbo_{j_l, f_l} \circ \dots \circ \cbo_{j_1, f_1} (0)$
$(l \in \{1, \dots, L\});$
then
$M = M^{(L)}$
and
$M^{(l)} = \cbo_{j_l, f_l}(M^{(l-1)})$
$(l \in \{1, \dots, L\}).$
By recursion, the number of entries of
$M^{(l)}$
(which is equal to the product of its dimensions) will be
$F_l := \prod_{s=1}^l f_s.$
\par
Let
$A_j^{(l)}$
be the set of entries on the $j$-th coordinate axis of
$M^{(l)},$
for any
$j \in \{1, \dots, m\}$
and
$l \in \{0, \dots, L\}.$
Then we find that
$A_j^{(0)} = \{0\}$
for all
$j \in \{1, \dots, m\},$
and that
$A_j^{(l)} = A_j^{(l-1)}$
if
$j \neq j_l,$
and
\begin{align}
 A_{j_l}^{(l)} &= A_{j_l}^{(l-1)} \cup (A_{j_l}^{(l-1)} + F_{l-1}) \cup (A_{j_l}^{(l-1)} + 2 F_{l-1}) \cup \cdots \cup (A_{j_l}^{(l-1)} + (f_l-1) F_{l-1})
\nonumber\\
 &= A_{j_l}^{(l-1)} + \left(\prod_{s=1}^{l-1} f_s \right) \ap{f_l};
\label{essunion}\end{align}
note that, since
$M^{(l-1)}$
is a principal reversible cuboid,
\begin{align}
 F_{l-1} &= \prod_{s=1}^{l-1} f_s = \sum_{r=1}^m \max A_r^{(l-1)} + 1 > \max A_{j_l},
\nonumber\end{align}
so the union in (\ref{essunion}) is a union of disjoint sets.
The formula (\ref{essbuild}) follows by recursion, as
$A_j = A_j^{(L)}.$
\end{proof}
We conclude with some examples to illustrate the workings of Theorem \ref{tssbuild} to construct
sum systems, and further the use of Theorems \ref{tSdsSs1}, \ref{tSdsSs2} to obtain corresponding
sum-and-distance systems.
\par\medskip
{\it Example 1.\/}
Take
$m = 3,$
$n = (15, 8, 6)$
and consider the joint ordered factorisation
\begin{align}
 ((1,5),(2,2),(1,3),(3,3),(2,2),(3,2),(2,2));
\nonumber\end{align}
then, by equation (\ref{essbuild}), we find the corresponding sum system
\begin{align}
 A_1 = &\{0,1,2,3,4,10,11,12,13,14,20,21,22,23,24\},
\nonumber\\
 A_2 = &\{0,5,90,95,360,365,450,455\},
\nonumber\\
 A_3 = &\{0,30,60,180,210,240\},
\nonumber\end{align}
which generates all integers
$0, 1, \dots, 6!-1$
exactly once.
Rearranging the same factors in a different joint ordered factorisation,
\begin{align}
 ((1,5),(3,3),(2,2),(3,2),(2,2),(1,3),(2,2)),
\nonumber\end{align}
we obtain a different sum system with component sets of the same cardinalities
$n_1, n_2, n_3,$
and the same target set,
\begin{align}
 A_1 = &\{0,1,2,3,4,120,121,122,123,124,240,241,242,243,244\},
\nonumber\\
 A_2 = &\{0,15,60,75,360,375,420,435\},
\nonumber\\
 A_3 = &\{0,5,10,30,35,40\}.
\nonumber\end{align}
\par\medskip
{\it Example 2.\/}
Consider
$m = 3$
and
$n = (14,8,6),$
all even. Then the joint ordered factorisation
\begin{align}
 ((1,2),(3,3),(2,2),(3,2),(2,2),(1,7),(2,2))
\nonumber\end{align}
gives the corresponding sum system
\begin{align}
 \tilde A_1 &= \{0,1,48,49,96,97,144,145,192,193,240,241,288,289\},
\nonumber\\
 \tilde A_2 &= \{0,6,24,30,336,342,360,366\},
\nonumber\\
 \tilde A_3 &= \{0,2,4,12,14,16\},
\nonumber\end{align}
and hence, by Theorem \ref{tSdsSs1}, the non-inclusive sum-and-distance system
\begin{align}
 A_1 = \{1, 95, 97, 191, 193, 287, 289\},
 A_2 = \{306, 318, 354, 366\},
 A_3 = \{8, 12, 16\}.
\nonumber\end{align}
\par\medskip
{\it Example 3.\/}
Consider
$m = 3$
and
$n = (15, 7, 9),$
all odd. Then the joint ordered factorisation
\begin{align}
 ((1,5),(2,7),(3,3),(1,3),(3,3))
\nonumber\end{align}
generates the sum system
\begin{align}
 \tilde A_1 &= \{0,1,2,3,4,105,106,107,108,109,210,211,212,213,214\},
\nonumber\\
 \tilde A_2 &= \{0,5,10,15,20,25,30\},
\nonumber\\
 \tilde A_3 &= \{0,35,70,315,350,385,630,665,700\},
\nonumber\end{align}
and further, by Theorem \ref{tSdsSs2}, the inclusive sum-and-distance system
\begin{align}
 A_1 = \{1, 2, 103, 104, 105, 106, 107\},
 A_2 = \{5, 10, 15\},
 A_3 = \{35, 280, 315, 350\}.
\nonumber\end{align}
\par\medskip
{\it Example 4.\/}
For
$m = 5$
and
$n = (28,20,30,18,12),$
the joint ordered factorisation
\begin{align}
 ((1,7),(2,4),(5,2),(3,2),(4,2),(2,5),(4,9),(3,3),(1,4),(5,3),(3,5),(5,2))
\nonumber\end{align}
gives, by formula (\ref{essbuild}), the five-part sum system
\begin{align}
 A_1 &=
 \{0,1,2,3,4,5,6,30240,30241,30242,30243,30244,30245,30246, 60480,
\nonumber\\
 &\qquad 60481,60482,60483,60484,60485,60486,90720,90721,90722,90723,90724,
\nonumber\\
 &\qquad 90725,90726\},
\nonumber\\
 A_2 &=
 \{0,7,14,21,224,231,238,245,448,455,462,469,672,679,686,693,896,903,
\nonumber\\
 &\qquad 910,917\},
\nonumber\\
 A_3 &=
 \{0,56,10080,10136,20160,20216,362880,362936,372960,373016,383040,
\nonumber\\
 &\qquad 383096,725760,725816,735840,735896,745920,745976,1088640,1088696,
\nonumber\\
 &\qquad 1098720,1098776,1108800,1108856,1451520,1451576,1461600,1461656,
\nonumber\\
 &\qquad 1471680,1471736\},
\nonumber\\
 A_4 &=
 \{0,112,1120,1232,2240,2352,3360,3472,4480,4592,5600,5712,6720,6832,
\nonumber\\
 &\qquad 7840,7952,8960,9072\},
\nonumber\\
 A_5 &=
 \{0,28,120960,120988,241920,241948,1814400,1814428,1935360,1935388,
\nonumber\\
 &\qquad 2056320,2056348\},
\nonumber\end{align}
which generates the integers
$0, 1, 2, \dots, 10!-1,$
each exactly once.

\end{document}